\RequirePackage{fix-cm}
\documentclass[smallextended]{svjour3}       
\smartqed  
\usepackage{graphicx}
\usepackage {color}
\smartqed  
\usepackage{graphicx}
\usepackage{amsmath}
\usepackage{amssymb}
\usepackage{mathtools}
\def\ST{\songti\rm\relax}
\def\R{{\mathbb R}}

\def\ST{{\rm s.t.}}

\begin{document}

\title{Polyak's convexity theorem, Yuan's lemma and S-lemma: extensions and applications\thanks{This research was supported by
	the Beijing Natural Science Foundation under grant Z180005, and
	the National Natural Science Foundation of China under grants 11822103 and 11771056.}}


\titlerunning{On Polyak's theorem, Yuan's lemma and S-lemma}        

\author{Mengmeng Song        \and
        Yong Xia 
}

\institute{Mengmeng Song \and Yong Xia (corresponding author) \at LMIB of the Ministry of Education, School of Mathematical Sciences, Beihang University, Beijing 100191, People's Republic of China.\\
	\email{\{songmengmeng,\ yxia\}@buaa.edu.cn}
	  }     

\date{Received: date / Accepted: date}
\maketitle

\begin{abstract}
We extend Polyak's theorem on the convexity of joint numerical range from three to any number of quadratic forms on condition that they can be generated by three quadratic forms with a positive definite linear combination. Our new result covers the fundamental Dines's theorem.
As applications, we further extend Yuan's lemma and S-lemma, respectively. Our extended Yuan's lemma is used to build a more generalized assumption than that of Haeser (J. Optim. Theory Appl. 174(3): 641-649, 2017), under which
the standard second-order necessary optimality condition holds at local minimizer. The extended S-lemma reveals strong duality of homogeneous quadratic optimization problem with two bilateral quadratic constraints.
\end{abstract}

\section{Introduction}

The convexity result on the joint numerical range of two quadratic forms in $\R^n$ dates back to Dines \cite{Dines41} in 1941. As a next milepost,
Brickman \cite{Brickman61} proves in 1961 that the joint numerical range of two quadratic forms over the unit sphere of  $\R^n\ (n\ge 3)$ is convex. In 1998, based on Brickman's result,  Polyak \cite{Polyak98} proves the convexity and closeness of the joint numerical range of three quadratic forms in $\R^n\ (n\ge 3)$, provided that there exists a positive definite linear combination of the three quadratic forms. Generally, the convexity of the joint numerical range of $m$ ($m>3$) quadratic forms remains unknown.
As is observed in \cite{Polyak98}, the existence of a linear positive definite combination of quadratic forms no longer guarantees the convexity for arbitrary $m> 3$. To our knowledge, the only known convexity result is that the Hessian matices of the $m$ quadratic forms can be simultaneously diagonalized in $\R^n$ \cite{fradkov1973}.

Convexity of the joint numerical range of quadratic forms has many applications, especially in establishing classical alternative results on  two quadratic forms including Yuan's lemma \cite{Yuan90} and  S-lemma \cite{polik2007,Yakubovich71}.
Both lemmas can be regarded as nonconvex extensions of the celebrated Farkas lemma and then proved by Dines's theorem \cite{Dines41}.
Yuan's lemma provides an equivalent characterization on the nonnegativity of  the maximum of two quadratic forms. It has applications in characterizing optimality conditions \cite{Baccari04} and minimax optimization \cite{haeser17}.  As an extension, Chen and Yuan \cite{chen1999} analyze  the nonnegativity of  the maximum of three quadratic forms and establish  a  necessary condition.
Haeser \cite{haeser17} extends Yuan's lemma to more than two quadratic forms, provided that the rank of the involved Hessian matrices is at most $2$.
The motivation inherits  from the extended S-lemma under the same assumption, see \cite[Proposition 3.5]{polik2007}. S-lemma has essential applications in control theory and robust optimization \cite{polik2007}.
As an application of Polyak's convexity theorem \cite{Polyak98}, Polyak himself generalizes in the same paper a novel S-lemma on three quadratic forms with a positive definite linear combination, see \cite[Theorem 4.1]{Polyak98}.

Our motivation in this paper dates back to Lao Tzu, an ancient Chinese philosopher, who said
\begin{center}
{\it
the three begets all things of the world.
}
\end{center}

We first extend Polyak's convexity theorem to arbitrary $m\ (\ge1)\ $ quadratic forms, provided that they can be generated by three quadratic forms with a positive definite linear combination.  It  covers not only Polyak's convexity theorem but also the fundamental Dines's theorem.

One application is to further extend Yuan's lemma to $m\ (\ge1)\ $ quadratic forms. It generalizes the classical Yuan's lemma and the extended Yuan's lemma due to Haeser \cite{haeser17} as special cases.
Yuan's lemma plays a great role in characterizing optimality conditions in nonlinear programming.
As Mangasarian-Fromovitz constraint qualification (MFCQ) \cite{Mangasarian67} cannot guarantee that the standard second-order necessary optimality condition holds at the local minimizer \cite{Anitescu00}, more assumptions have to be introduced.
Based on the classical Yuan's lemma, Baccari and Trad \cite{Baccari04} give such an assumption that the set of Lagrange multiplier is a bounded line segment.
Haeser \cite{haeser17} uses an extension of Yuan's lemma to generalize
Baccari and Trad's assumption to the case that the Hessian matrices of the Lagrangian, evaluated at the vertices of the Lagrange multiplier set, form a matrix set with rank at most two. Based on our
further extended Yuan's lemma, we can relax the rank assumption on the Hessian matrices set from two to three, provided the matrices have a positive definite linear combination. Our result generalizes not only Baccari and Trad's assumption  \cite{Baccari04} but also Haeser's assumption \cite{haeser17}.

The second application of the extended Polyak's convexity theorem is to further extend the homogeneous S-lemma and its extension \cite[Proposition 3.5]{polik2007},  and Polyak's S-lemma \cite[Theorem 4.1]{Polyak98}.
As an application, we establish the strong duality of the homogeneous quadratic optimization with two homogeneous quadratic  bilateral constraints, which includes the generalized trust region subproblem with interval bounds and the two-side regularized total least squares problem as special cases.

The remainder of this  paper is arranged as follows. Section 2 extends Polyak's convexity theorem to $m\ (\ge1)$ quadratic forms on condition that their Hessian matrices can be generated by three matrices with a positive definite linear combination. In Section 3, Yuan's lemma and its extension are further extended. We then establish a more generalized assumption so that the standard second-order necessary optimality condition holds at local minimizer. In  Section 4, we further extend the homogeneous S-lemma and its extension. As an application, we establish strong duality of the homogeneous quadratic programming with two quadratic bilateral constraints, which includes the generalized trust region subproblem with interval bounds and the two-side regularized total least squares problem as special cases. We make a conclusion in Section 5.

{\bf Notations.}
Let $\R_+^m\ (\R_-^m)$ and $\mathcal{S}^n$ be the nonnegative (nonpositive) orthant in $\R^m$ and the set of symmetric matrices in $\R^{n\times n}$, respectively. Let $\Delta_m=\{t\in\R_+^m:\sum_{i=1}^m t_i=1\}$ be the $(m-1)$-dimensional simplex.
Denote by $I$ the identity matrix of proper dimension. Let $A\succ(\succeq)$ denotes that $A$ is positive (semi-) definite.
For any set $K\subset\R^n$, let ${\rm span}(K)$ be the subspace spanned by $K$ and  ${\dim}(K)$ be the dimension of ${\rm span}(K)$. Denote by ${\rm int}K$ and ${\rm ri}K$ the interior and relative interior of $K$, respectively.
$K\subset\R^n$ is a cone, if for any $x\in K$, $\lambda x\in K$ holds for all $\lambda\ge 0$. The cone $K\subset\R^n$ is  acute if there is no $x\neq 0$ such that $x, -x\in K$.
For any twice continuously differentiable function $g: \R^n\rightarrow\R$ and $x\in\R^n$, let $\nabla g(x)$ and $\nabla^2g(x)$ be the gradient and the Hessian matrix of $g$ at $x$, respectively. Denote by $v(\cdot)$ the optimal value of the problem $(\cdot)$.
%

\section{Extended Polyak's theorem}
The following Polyak's theorem \cite[Theorem 2.1]{Polyak98} reveals the convexity of  the joint numerical range of three quadratic forms with a positive definite linear combination.

\begin{theorem}[Polyak's theorem]\label{le:ployak}
Let $n\ge 3$ and $A_1, A_2, A_3\in\mathcal{S}^n$. Suppose that there exists $s\in\R^3$ such that $s_1A_1+s_2A_2+s_3A_3\succ 0$. Then, the joint numerical range $\{(x^TA_1x, x^TA_2x,x^TA_3x):x\in\R^n\}$ is an acute closed convex cone.
\end{theorem}	

In this section, we study the joint numerical range of $m\ (m\ge1)$ matrices.
We notice that Polyak's theorem fails to extend to the case $m\ge 4$ even when the $m$ matrices have a positive definite linear combination, see the counterexample \cite[Example 3.7]{Polyak98}.

We begin with  the definition of matrices generation.
\begin{definition}\label{de:ge}
The set of matrices $\mathcal{B}=\{B_i: i=1,\cdots, m\}$ is generated by $\mathcal{A}=\{A_j:j=1,\cdots,k\}$, if there exist $r_{ij}\in\R~ (i=1,\cdots,m,\ j=1,\cdots,k)$ such that $B_i=\sum_{j=1}^k r_{ij} A_j$ for $i=1,\cdots,m$.
\end{definition}

Definition \ref{de:ge} coincides with  ``linear expression'' in linear algebra, if we regard $\mathcal{S}^n$ as a linear space of dimension $n(n+1)/2$. We define the rank of  $\mathcal{B}$ as
\[
{\rm rank} (\mathcal{B})=
\min\{k:~\mathcal{B} {\rm\ can\ be\ generated\ by\ } \{B_{i_1},\cdots,B_{i_k}\} \},
\]
which is equal to the dimension of the subspace generated by $\mathcal{B}$. Moreover, not only the rank of $\mathcal{B}$ but also the corresponding index set $\{ i_1,\cdots, i_k\}$ can be efficiently obtained by Gaussian elimination. $\mathcal{A}$ is called a basis of $\mathcal{B}$, if $\mathcal{B}$ can be generated by $\mathcal{A}$, and $\mathcal{A}$ is a set of ${\rm rank} (\mathcal{B})$ matrices.

%
%

Our main result is to show that the joint numerical range of $m$ quadratic forms is convex on condition that the $m$ Hessian matrices can be generated by three matrices with a positive definite linear combination.
\begin{theorem}\label{th:ex2}
Suppose that $n\ge 3,~m\ge 1$,  $\{B_i: i=1,\cdots, m\}\subset\mathcal{S}^{n}$ can be generated by $\{A_1, A_2, A_3\}\subset\mathcal{S}^{n}$, and there exists a vector $s\in\R^3$ such that $s_1A_1+s_2A_2+s_3A_3\succ0$. Then, \begin{equation}
\{(x^TB_1x, x^TB_2x,  \cdots,x^TB_mx):~x\in\R^n\} \label{Pset}
\end{equation}
 is a  convex cone.
\end{theorem}
\begin{proof}
	%
	%
Let
\begin{eqnarray*}
&&\Omega_0=\{(x^TA_1x, x^TA_2x, x^TA_3x):~x\in\R^n\},\\
&&\Omega=\{(x^TB_1x, x^TB_2x,\cdots, x^TB_mx):~x\in\R^n\}.
\end{eqnarray*}
Since $\{B_i: i=1,\cdots, m\}$ can be generated by $\{A_1, A_2, A_3\}$, there exist $r_{ij}~(i=1,\cdots,m,\ j=1,2,3)$ such that \[
B_i=r_{i1}A_1+r_{i2}A_2+r_{i3}A_3,~i=1,\cdots, m.
\]
Let $Q=(r_{ij})\in\R^{m\times3}$. Then, we have
\begin{equation}\label{eq:Omega}
	\Omega=\Omega_0Q^T.
\end{equation}
According to Theorem \ref{le:ployak}, $\Omega_0$ is a convex cone. Under linear transformation \eqref{eq:Omega}, $\Omega$ remains a convex cone.\qed
\end{proof}	

\begin{remark}\label{re:1}
The basis of $\mathcal{B}=\{B_i:~i=1,\cdots, m\}$ is not unique. However, if the positive definite linear combination assumption is satisfied for one basis,  it holds for any other basis.

Suppose the rank of $\mathcal{B}$ is $k$. Let   $\mathcal{A}=\{A_j:~j=1,\cdots,k\}$ and $\bar{\mathcal{A}}=\{\bar A_j:~j=1,\cdots,k\}$ be two bases of $\mathcal{B}$. Suppose that $\sum_{j=1}^{k}s_jA_j\succ0$ holds for some $s\in\R^k$, we show in the following that there exists $s'\in\R^k$ such that $\sum_{j=1}^{k}s'_j\bar A_j\succ0$ holds.

Since ${\rm span}(\mathcal{A})={\rm span}(\bar{\mathcal{A}})={\rm span}(\mathcal{B})$,
$\mathcal{A}$ can be generated by $\bar{\mathcal{A}}$. That is, there exist $a_{i,j} \ (i, j=1, \cdots, k)$ such that
\[A_i=\sum_{j=1}^{k}a_{i,j}\bar A_j\ {\rm for}\ i=1, \cdots, k.\]
Then, it holds that
\[
0\prec\sum_{i=1}^{k}s_iA_i=\sum_{i=1}^{k}s_i\sum_{j=1}^{k}a_{i,j}\bar A_j=\sum_{j=1}^{k}\left(\sum_{i=1}^{k}s_ia_{i,j}\right)\bar A_j.
\]
Letting $s'_j=\sum_{i=1}^{k}s_ia_{i,j}$ yields that $\sum_{j=1}^{k}s'_j\bar A_j \succ 0$.
\end{remark}

\begin{remark}\label{re:P}
Theorem \ref{th:ex2} covers Dines's theorem \cite{Dines41}, which establishes the convexity of
\[
\{(x^TA_1x, x^TA_2x):x\in\R^n\}
\]
for any  $A_1,A_2\in\mathcal{S}^n$ and $n\ge 1$. As a comparison, Theorem \ref{le:ployak} cannot completely cover Dines's theorem.

Actually, let $B_3=I\in \mathcal{S}^{n+2}$,
\[
B_1=\bmatrix A_1 & 0 & 0\\ 0 &  0 & 0\\ 0 &  0 & 0
	\endbmatrix\in \mathcal{S}^{n+2},~
B_2=\bmatrix A_2 & 0 & 0\\ 0 &  0 & 0\\ 0 &  0 & 0
	\endbmatrix\in \mathcal{S}^{n+2}.
\]
We have $0\cdot B_1+0\cdot B_2+ 1\cdot I\succ 0$.
As $\{B_1, B_2\}\subset\mathcal{S}^n$
can be generated by $\{B_1,B_2,B_3\}$, according to  Theorem \ref{th:ex2}, we have
\[
\{(y^TB_1y, y^TB_2y):~y\in\R^{n+2}\}=
\{(x^TA_1x, x^TA_2x):~x\in\R^n\}
\]
is a convex cone.
\end{remark}

The convexity in Theorem \ref{th:ex2} is inherited from Lemma \ref{le:ployak}, but the closeness and acuteness cannot be retained.
\begin{remark}
Different from Theorem \ref{le:ployak}, the set \eqref{Pset} may be not acute under the condition of Theorem \ref{th:ex2}.
The following counterexample is motivated by  \cite[Example 3]{Dines41}. Let
\[
	B_1=\bmatrix 1 & 1 & 0\\ 1 &  0 & 0\\ 0 &  0 & 0
	\endbmatrix, ~
	B_2=\bmatrix 0 & 1 & 0\\ 1 & 1 & 0\\ 0 & 0 & 0
	\endbmatrix,
\]
then the set
\[\{(x^TB_1x, x^TB_2x): ~x\in\R^3\}=\{(x_1^2+2x_1x_2,  x_2^2+2x_1x_2):~x_1, x_2\in\R\}=\R^2
\]
is certainly not acute.
\end{remark}

\begin{remark}
Different from Theorem \ref{le:ployak}, the set \eqref{Pset} may be not closed under the condition of Theorem \ref{th:ex2}.
The following counterexample is motivated by \cite{Juan05}.    Let
\[
B_1=\bmatrix 1 & 0 & 0\\ 0 &  -1 & 0\\ 0 &  0 & 0
	\endbmatrix,~
B_2=\bmatrix 2 & -1 & 0\\ -1 & 0 & 0\\ 0 & 0 & 0
	\endbmatrix,
\]
and $M=\{(x^TB_1x, x^TB_2x): ~x\in\R^3\}$.
We claim that $(1,1)\in{\rm cl}(M)$ and $(1,1)\notin M$.
It can be first verified that
\[
x^TB_2x-x^TB_1x=2x_1^2-2x_1x_2-(x_1^2-x_2^2)={(x_1-x_2)}^2\ge 0.
\]
Therefore, if $x^TB_2x=x^TB_1x$, then we have $x_1=x_2$ and hence $x^TB_2x=x^TB_1x=0$. That is,  $(1,1)\notin M$.
Define the sequence
\[
x(k)=\left(\frac{1}{\sqrt{2}}(k+\frac{1}{k}), \frac{1}{\sqrt{2}}k, 0\right)^T,~ k=1, 2, \cdots,
\]
we have
\[
\lim\limits_{k\rightarrow+\infty}x(k)^TB_ix(k)=1, ~i=1, 2.
\]
Thus, $(1,1)\in{\rm cl}(M)$ and $M$ is not closed.
\end{remark}
\section{Further extended Yuan's lemma and its application}

Based on Theorem \ref{th:ex2}, we further extend Yuan's lemma. As an application, we establish a more generalized assumption to guarantee the standard second-order necessary condition  for local minimizer.

We start from the classical Yuan's lemma \cite[Lemma 2.3]{Yuan90}.
\begin{theorem}[Yuan's lemma]\label{le:y}
For any $A_1, A_2\in\mathcal{S}^n$, the following two assertions are equivalent:
\begin{itemize}	
\item[(a)] $\forall x\in\R^n:~\max\{x^TA_1x, x^TA_2x\}\ge0$.
\item[(b)] There exists $t\in\Delta_2$ such that
	$t_1A_1+t_2A_2\succeq 0$.
\end{itemize}	
\end{theorem}
Yuan's lemma has been extended from $\R^n$ to the first-order cones by Baccari and Trad \cite{Baccari04}, and then to the regular cones introduced by Jeyakumar et al. \cite{jeyakumar09}. For completeness, we present their definitions.
\begin{definition}
The cone $K\subset\R^n$ is a first-order cone, if there exists a subspace $S$ and a vector $d\in\R^n$ such that  $K=S+\R_+d$ holds. $K$ is a regular cone, if $K\cup-K$ is a subspace of $\R^n$.
\end{definition}
According to the above definitions, any first-order cone (including $\R^n$) is a regular cone.

Recently, Haeser extended Yuan's lemma to $m\ (\ge 2)$ quadratic forms that can be generated by two quadratic forms, see \cite[Lemma 2.2]{haeser17}.
\begin{theorem}[Extended Yuan's lemma \cite{haeser17}]\label{le:y2}
Let $A_i\in\mathcal{S}^n ~(i=1, \cdots, m)$ be such that
$A_i=\alpha_iA_1+\beta_iA_2 $ for some $(\alpha_i, \beta_i)\in\R^2, ~i=3, \cdots, m$, and $K\subseteq\R^n$ be a first-order cone. Then, the following are equivalent to each other:
\begin{itemize}
\item[(c)] $\forall x\in K: ~\max_{i=1,\cdots,m}\{x^TA_ix\}\ge 0.$
\item[(d)] $\exists\  t\in \Delta_m$ such that
$\forall x\in K:\ \sum_{i=1}^mt_ix^TA_ix\ge 0$.
\end{itemize}
\end{theorem}	
%
%
%
%

As a main result in this section, we further extend Theorem \ref{le:y2} to $m\ (\ge 1)$ quadratic forms, provided that they can be generated by three quadratic forms with a positive definite linear combination.

\begin{theorem}\label{th:kkkkkk}
Let $K$ be a regular cone with ${\dim}(K)\ge 3$, and $\{B_i: i=1,\cdots, m\}$ $(m\ge 1)$ can be generated by $\{A_1, A_2, A_3\}$ satisfying
\begin{equation}\label{eq:K_s}
\forall 0\neq x\in K:\  s_1x^TA_1x+s_2x^TA_2x+s_3x^TA_3x>0
\end{equation}
for some $s_1, s_2, s_3\in\R$.	Then, the following are equivalent to each other:
\begin{itemize}
\item[(e)] $\forall x\in K:\  \max\{x^TB_1x,\cdots,x^TB_mx\}\ge0$.
\item[(f)] $\exists\  t\in\Delta_m$ such that $\forall x\in K:\ \sum_{i=1}^mt_ix^TB_ix\ge 0$.
\end{itemize}
\end{theorem}
\begin{proof}	
Proof of $(e)\rightarrow(f)$.
We first assume that $K=\R^n$.
Let \[\Omega=\{(x^TB_1x, x^TB_2x,\cdots,x^TB_mx):~x\in \R^n\}.\] It holds from $(e)$ that ${\rm int}\R_-^m\cap \Omega=\emptyset$. According to Theorem \ref{th:ex2}, $\Omega$ is a  convex cone. Then, according to the well-known   convex set separation theorem, there exists $t\in\R^m\setminus\{0\}$ such that
\begin{eqnarray}
	&&t^Ty\le0,\ \forall y\in \R_-^m,\label{eq:S}\\
	&&t^Ty\ge0,\ \forall y\in \Omega.\label{eq:Q}
\end{eqnarray}
It follows from the definition of the set $\R_-^m$ and \eqref{eq:S} that $t\in\R_+^m$. Then $\sum_{i=1}^mt_i>0$.
Without loss of generality, we assume $t\in\Delta_m$ in \eqref{eq:S} and \eqref{eq:Q}, since otherwise, we can divide both sides by $\sum_{i=1}^mt_i$.
Then \eqref{eq:Q} implies that
\[
\forall x\in\R^n:\ \sum_{i=1}^{m}t_ix^TB_ix\ge0,
\]
or equivalently, $\sum_{i=1}^m t_iB_i\succeq 0$.

Now let $K$  be a regular cone. Then \eqref{eq:K_s} is equivalent to
\begin{eqnarray}
\forall 0\neq x\in K\cup  (-K):\  s_1x^TA_1x+s_2x^TA_2x+s_3x^TA_3x>0,\label{eq:q1}
\end{eqnarray}
and $(e)$ can be rewritten as
\begin{eqnarray}
\forall x\in K\cup  (-K):\  \max\{x^TB_1x,x^TB_2x,\cdots,x^TB_mx\}\ge0.\label{eq:q2}
\end{eqnarray}
Let $k=\dim(K)\ (\ge3)$. There exists $Q\in\R^{n\times k}$ such that
\begin{equation}
K\cup (-K)=\{Qy:~y\in\R^k\}. \label{K-K}
\end{equation}
By substituting \eqref{K-K}
into \eqref{eq:q1} and \eqref{eq:q2}, respectively, we obtain
\begin{eqnarray}
&&\forall 0\neq y\in\R^k:\  s_1y^TQ^TA_1Qy+s_2y^TQ^TA_2Qy+s_3y^TQ^TA_3Qy>0,\\
&&\forall y\in \R^k:\  \max\{y^TQ^TB_1Qy,y^TQ^TB_2Qy,\cdots,y^TQ^TB_mQy\}\ge0.
\end{eqnarray}
According to the above first part of this proof, there exists $  t\in\Delta_m$ such that
\[\forall y\in \R^k:\ \sum_{i=1}^mt_iy^TQ^TB_iQy\ge 0,\]
which is equivalent to
\[\forall x\in K:\ \sum_{i=1}^mt_ix^TB_ix\ge 0.\]

Proof of $(f)\rightarrow(e)$. Suppose, on the contrary,
there exists $\tilde x\in K$ such that $\tilde x^TB_i\tilde x<0$ holds for $i=1,\cdots,m$. Then there is no $t$ satisfying
\begin{equation*}
	t\ge0, \ t\neq0,\ \tilde x^T(t_1B_1+\cdots+t_mB_m)\tilde x\ge 0,
\end{equation*}
which is a contradiction to $(f)$.
\qed
\end{proof}
\begin{remark}
Based on an analysis similar to that in Remark \ref{re:P}, we can show that Theorem \ref{th:kkkkkk} covers not only the original Yuan's lemma (Theorem \ref{le:y}), but also the extended Yuan's lemma (Theorem \ref{le:y2}).
\end{remark}

We apply our further extended Yuan's lemma to the second-order necessary optimality condition in nonlinear programming.

Consider the following general nonlinear programming in $\R^n$:
\begin{equation}\label{eq:NLP}\tag{NLP}
	\begin{array}{cl}
		\min & f(x)\\
		\ST   &
		g_i(x)\le0,~i=1,\cdots,p, \\
		& h_j(x)=0,~j=1,\cdots,q,\\
	\end{array}
\end{equation}
where all involved functions are twice continuously differentiable.


%

Firstly, we review some definitions and classical results on the second-order necessary optimality condition, which can be found in \cite{Bonnans00}.

\begin{itemize}
\item [(1)] Denote by $I(x)$ the set of  inequality constraints active at $x$, i.e.,
	\[I(x):=\{i\in\{1,\cdots,p\}:~g_i(x)=0\}.\]
\item [(2)] The Lagrangian and the generalized Lagrangian associated with  \eqref{eq:NLP} are defined as
\begin{eqnarray*}
&&L(x,\lambda,\mu):= f(x)+\sum_{i=1}^{p}\lambda_ig_i(x)+\sum_{j=1}^{q}\mu_jh_j(x),\\ &&L^{g}(x,\lambda_0,\lambda,\mu):= \lambda_0f(x)+\sum_{i=1}^{p}\lambda_ig_i(x)+\sum_{j=1}^{q}\mu_jh_j(x),
\end{eqnarray*}
respectively.
	\item [(3)]	The set $\Lambda(\bar x)$ of Lagrange multipliers at a feasible point $\bar x$ of \eqref{eq:NLP} is defined as the set of nonzero vectors $(\lambda, \mu)$ satisfying the following first-order necessary optimality conditions:
	\[\nabla_x L(\bar x,\lambda,\mu)=0;\  \lambda_i\ge0,\ \lambda_ig_i(\bar x)=0,  ~i=1,\cdots,p.\]
	\item [(4)]	The set $\Lambda^g(\bar x)$ of generalized Lagrangian multiplier at the feasible point $\bar x$ of \eqref{eq:NLP} is defined as the set of nonzero vectors $(\lambda_0, \lambda, \mu)$ satisfying the following necessary optimality conditions:
	\[\nabla_x L^{g}(\bar x,\lambda_0,\lambda,\mu)=0;\  \lambda_0\ge0;\ \lambda_i\ge0,\ \lambda_ig_i(\bar x)=0,~ i=1,\cdots,p.\]
For a generalized Lagrange multiplier $(\lambda_0, \lambda, \mu)$ with $\lambda_0=0$, we call $(\lambda, \mu)$ a singular Lagrange multiplier.
	\item [(5)] The critical cone associated with a feasible point $\bar x$ of $\eqref{eq:NLP}$ is
	\begin{equation}\label{eq:critical cone}
\begin{array}{cl}
		C(\bar x):=\{v: & \nabla f(\bar x)^Tv\le 0,\ \nabla g_i(\bar x)^Tv\le 0, ~ i\in I(\bar x);\\
		 &
		 \nabla h_j(\bar x)^Tv=0, ~j=1,\cdots,q\}.
	\end{array}
	\end{equation}
\end{itemize}
By the definition of the Lagrange multiplier, $\Lambda(\bar x)$ is a closed polyhedral convex set, since it can be expressed as the intersection of finite collection of closed half-spaces.
And by the definition of the generalized Lagrange multiplier, $\Lambda^g(\bar x)$ is a convex cone.
If a generalized Lagrange multiplier $(\lambda_0, \lambda, \mu)$ satisfies  $\lambda_0=1$, then $(\lambda, \mu)\in\Lambda(\bar x)$ is a Lagrange multiplier.

The first lemma could help to characterize the relation between the singular Lagrange  multiplier and the Lagrange multiplier.
\begin{lemma}[Proposition 3.14 in \cite{Bonnans00}]\label{re:singular}
If the set of Lagrange multiplier $\Lambda(\bar x)$ is nonempty, then the set of singular Lagrange multiplier, together with $0$, forms the recession cone of $\Lambda(\bar x)$.
\end{lemma}
According to Lemma \ref{re:singular}, we have the following result.
\begin{corollary}\label{co:soc1}
The set of Lagrange multiplier $\Lambda(\bar x)$ is bounded if and only if the set of singular Lagrange  multiplier is empty.
\end{corollary}

The following
Fritz-John second-order necessary optimality condition holds
without any constraint qualification (CQ).
\begin{lemma}[Proposition 5.48 in \cite{Bonnans00}]\label{le:Bonnans}
If $\bar x$ is a locally optimal solution of \eqref{eq:NLP}, then for every $v\in C(\bar x)$, there exists a generalized Lagrange multiplier $(\lambda_0,\lambda,\mu)\in\Lambda^g(\bar x)$ such that
\begin{equation}\label{eq:fj_ne}
v^T\nabla^2L^g(\bar x, \lambda_0,\lambda,\mu)v\ge 0.
\end{equation}
\end{lemma}
As shown in Lemma \ref{le:Bonnans}, the generalized Lagrange multiplier satisfying \eqref{eq:fj_ne} may depend on the choice of $v\in C(\bar x)$.
Based on the inconsistency of the generalized Lagrange multiplier in terms of $v\in C(\bar x)$,  we are interested in when the second-order necessary optimality condition can be
verified by a single Lagrange multiplier.

\begin{definition}
Let $\bar x$ be a local minimizer of \eqref{eq:NLP}. The standard second-order necessary optimality condition holds at $\bar x$, if there exists a Lagrange multiplier $(\lambda,\mu)\in\Lambda(\bar x)$ such that
\[ \forall v\in C(\bar x):\ v^T\nabla^2L(\bar x, \lambda,\mu)v\ge 0.
\]
\end{definition}	
The usual approach to avoid the inconsistency of the generalized Lagrange multiplier and hence strengthen Lemma \ref{le:Bonnans}, is to add qualification  constraints, for example, the following well-known linear independence constraint qualification (LICQ).
\begin{definition}
LICQ  holds at $\bar x$, a feasible point of \eqref{eq:NLP}, if the vectors $ \nabla g_i(\bar x)~(i\in I),\ \nabla h_j(\bar x)~(j=1,\cdots,q)$ are linearly independent, that is,  the linear system
	\begin{equation}\label{eq:zero}
		\sum_{i\in I}\lambda_i\nabla g_i(\bar x)+\sum_{j=1}^{p}\mu_j\nabla h_j(\bar x)=0
	\end{equation}
in terms of $(\lambda,\mu)$ has no nonzero solution.
\end{definition}
LICQ ensures not only that the singular Lagrangian multiplier does not exist, but also that the set of Lagrangian multiplier is a single-point set. Therefore, it follows from Lemma \ref{le:Bonnans} that the standard second-order necessary optimality condition holds true under LICQ.

However, LICQ could fail to hold at local minimizer of \eqref{eq:NLP}. Mangasarian-Fromovitz constraint qualification (MFCQ), a more relaxed well-known CQ, is first proposed in \cite{Mangasarian67}.
\begin{definition}
MFCQ holds at $\bar x$, if the following conditions are satisfied:
\begin{itemize}
\item[(i)] The vectors $\nabla h_j(\bar x), j=1,\cdots,q$ are linearly independent.
\item[(ii)] $\exists  u\in\R^n$:~$\nabla g_i(\bar x)^Tu<0,~i\in I(\bar x);\ \nabla h_j(\bar x)^Tu=0, ~j=1,\cdots,q.$
\end{itemize}
\end{definition}
The following equivalent characterization of MFCQ is due to  Gauvin \cite{gauvin77}.
\begin{theorem}{\rm \cite{gauvin77}}\label{th:e_MFCQ}
Let $\bar x$ be a local minimizer of \eqref{eq:NLP}. Then MFCQ holds at $\bar x$ if and only if the set of Lagrange multipliers $\Lambda(\bar x)$ is nonempty and bounded.
\end{theorem}
According to Theorem \ref{th:e_MFCQ}, we can strengthen Lemma \ref{le:Bonnans} under MFCQ.
\begin{lemma}\label{co:soc}
Assume that MFCQ holds at the feasible point $\bar x$ of \eqref{eq:NLP}.
If $\bar x$ is a local minimizer, then for every $v\in C(\bar x)$, there exists a Lagrange multiplier $(\lambda,\mu)\in\Lambda(\bar x)$ such that
	\[v^T\nabla^2L(\bar x, \lambda,\mu)v\ge 0.\]
\end{lemma}	
However, MFCQ itself cannot ensure the standard second-order necessary optimality condition, see a counterexample in \cite{Anitescu00}.
So additional assumptions are required as in \cite{Baccari04,haeser17}. In the following, we propose a new assumption together with MFCQ to guarantee the standard second-order necessary optimality condition. It  generalizes the assumptions presented in \cite{Baccari04,haeser17}.
\begin{theorem}\label{th:main}
Let $x^*$ be a local minimzer of \eqref{eq:NLP} and MFCQ holds at $x^*$. Let $(\lambda^1,\mu^1), (\lambda^2,\mu^2), \cdots, (\lambda^m,\mu^m)$ be the vertices of the Lagrange multiplier set $\Lambda(x^*)$. If there exist $A_{1}, A_{2}, A_{3}\in\mathcal{S}^n$ such that $\{{\nabla^2 L(x^*,\lambda^i,\mu^i)}: i=1, \cdots, m\}$ can be generated by $\{A_{1}, A_{2}, A_{3}\}$ and $s_{1}A_{1}+s_{2} A_{2}+s_{3}A_{3}\succ0$ holds for some $s_{1}, s_{2}, s_{3}\in\R$.
Then, for every regular cone $K\subseteq C(x^*)$ with ${\rm dim}K\ge 3$, there exists $(\lambda,\mu)\in\Lambda(x^*)$ such that
	\begin{equation}\label{eq:lm}
		\forall v\in K:\ v^T\nabla^2 L(x^*,\lambda,\mu)v\ge0.
	\end{equation}
\end{theorem}
\begin{proof}
According to Lemma \ref{co:soc}, it holds that
\begin{equation}
\forall v\in K:\ \max\{v^T\nabla^2 L(x^*,\lambda,\mu)v: ~ (\lambda,\mu)\in\Lambda(x^*)\}\ge0.\label{eq:vv0}
\end{equation}
Since the set $\Lambda(x^*)$ has $m$ vertices and $v^T\nabla^2 L(x^*,\lambda,\mu)v$ is linear with respect to $(\lambda,\mu)$, \eqref{eq:vv0} is equivalent to
\[
\forall v\in K:\ \max\{v^T\nabla^2 L(x^*,\lambda^i,\mu^i)v:~ i=1,\cdots,m\}\ge0.
\]
According to Theorem \ref{th:kkkkkk}, there exists
	$t\in\Delta_m$ such that 	\eqref{eq:lm} holds with
	$(\lambda,\mu)=\sum_{i=1}^mt_i(\lambda^i,\mu^i)\in \Lambda(x^*)$.\qed
\end{proof}
\begin{remark}
If $C(x^*)$ itself is a regular cone, the standard second-order necessary optimality condition holds at $x^*$.
One sufficient condition to ensure $C(x^*)$ to be a regular cone is the generalized strict complementarity slackness, see  \cite{Baccari04} for more details.
\end{remark}
\begin{remark}
Theorem \ref{th:main} covers the main result in \cite[Theorem 3.2]{haeser17}, which assumes that the rank of $\{{\nabla^2 L(x^*,\lambda^i,\mu^i)}: ~i=1, \cdots, m\}$ is at most $2$. Suppose $\{{\nabla^2 L(x^*,\lambda^i,\mu^i)}: ~i=1, 2, \cdots, m\}$ is generated by two matrices $\{A_1, A_2\}$. Then it can also be generated by $\{A_1, A_2, I\}$. Moreover, letting $s_1=s_2=0, s_3=1$ implies that the positive definite linear combination assumption in Theorem \ref{th:main} holds.

\end{remark}

\section{A further extended S-lemma and its application}
The celebrated S-Lemma is due to \cite{Yakubovich71}, see \cite{polik2007} for a survey. In this section, based on Theorem \ref{th:ex2}, we further extend S-lemma. It covers the homogeneous S-lemma and its several extensions. Then we present its applications in revealing hidden convexity of quadratic optimization.

\begin{theorem}[Homogeneous S-lemma]\label{le:hs}
For $A_0, A_1\in\mathcal{S}^n, \beta_0, \beta_1\in\R$, suppose that there exists $\bar x\in\R^n$ such that $\bar x^TA_1\bar x<\beta_1$. The following two assertions are equivalent:
\begin{itemize}
\item[(a)] $x^TA_1x\le \beta_1~\Longrightarrow~x^TA_0x\ge\beta_0$.
\item[(b)]   $\exists\ t\ge0$ such that $A_0+tA_1\succeq 0,\  \beta_0+t\beta_1\le0.$
\end{itemize}
\end{theorem}	
As presented in \cite[Proposition 3.5]{polik2007}, Theorem \ref{le:hs} is extended to $m\ (\ge1)$ quadratic forms generated by at most two quadratic forms. We remark that the original version \cite[Proposition 3.5]{polik2007} missed assuming Slater condition, which is necessary.
\begin{theorem}[Extended S-lemma]\label{le:slemma}
For $m\ge 1$ and $B_0, B_1, \cdots, B_m \in\mathcal{S}^n$, suppose that ${\rm rank} ({B_0, B_1, \cdots, B_m})\le 2.$ Under Slater condition, i.e., there exists $\bar x\in\R^n$ such that $\bar x^TB_i\bar x<0,~ i=1, \cdots, m$, the following two assertions are equivalent:
\begin{itemize}
\item[(c)] $x^TB_ix\le0,\  i=1, 2, \cdots, m~\Longrightarrow~x^TB_0x\ge0           $.	
\item[(d)]   $\exists\ t\ge0$ such that
$B_0+\sum_{i=1}^{m}t_iB_i\succeq 0.$
\end{itemize}
\end{theorem}	
We extend Theorems \ref{le:hs} and \ref{le:slemma} to $m\ (\ge1)$ quadratic forms that can be generated by three quadratic forms with a positive definite linear combination.
\begin{theorem} \label{th:g_slemma}
	Suppose that $n\ge 3, m\ge 1$ and $\{B_i: i=0, 1, \cdots, m\}$ can be generated by $\{A_1, A_2, A_3\}$ with $s_1A_1+s_2A_2+s_3A_3\succ0$ for some $s_1, s_2, s_3\in\R$. Under Slater condition, i.e., there exists $\bar x\in\R^n$ such that $\bar x^TB_i\bar x<\beta_i, i=1, \cdots, m$,  the following are equivalent:
\begin{itemize}
\item[(e)] $ x^TB_ix\le\beta_i, ~i=1, \cdots, m~\Longrightarrow~x^TB_0x\ge\beta_0$.	
\item[(f)] $\exists\ 0\le t\in\R^m$ such that
$
		B_0+\sum_{i=1}^m t_iB_i\succeq0,\
		\beta_0+\sum_{i=1}^m t_i\beta_i\le0.
$	
\end{itemize}
\end{theorem}
\begin{proof}
Proof of $(e) \rightarrow(f)$. Denote
\begin{eqnarray*}
 &&\Omega=\{(x^TB_0x-\beta_0, x^TB_1x-\beta_1, x^TB_2x-\beta_2,\cdots,x^TB_mx-\beta_m):~x\in \R^n\},\\
 &&\Omega_0=\{(y_0, y_1, y_2,\cdots, y_m): ~y_i\le 0,~ i=0, 1,\cdots,m\}.
\end{eqnarray*}
It follows from $(e)$ that ${\rm int}\Omega_0\cap \Omega=\emptyset$. According to Theorem \ref{th:ex2}, $\Omega$ is convex. Then, according to  the well-known convex set separation theorem, there exists $(t_0, t_1, \cdots, t_m)\in\R^{m+1}\setminus\{0\}$ such that
	\begin{eqnarray}
	&&\forall x\in\R^n:\  t_0	(x^TB_0x-\beta_0)+\sum_{i=1}^mt_i (x^TB_ix-\beta_i)\ge0,\label{eq:S1}\\
	&&\forall y \in\Omega_0:\  t_0y_0+\sum_{i=1}^m t_iy_i\le0.\label{eq:S2}
	\end{eqnarray}
By \eqref{eq:S2}, it holds that $t_i\ge0$ for $i=0, 1, \cdots, m$. If $t_0=0$,	taking $x=\bar x$ in \eqref{eq:S1} yields that $t_i=0, ~i=1,\cdots,m$, which contradicts the assumption $t\neq 0$. Thus, we have $t_0>0$. 
Dividing both sides of \eqref{eq:S1} by $t_0$ yields that
\[
\forall x\in\R^n:\ x^TB_0x+\sum_{i=1}^m\frac{t_i}{t_0} x^TB_ix\ge\beta_0+\sum_{i=1}^m\frac{t_i}{t_0}\beta_i,
\]
which completes the proof of $(f)$.

Proof of $(f) \rightarrow(e)$.  Suppose, on the contrary, $(e)$ does not hold. Then there exists $\tilde x\in\R^n$ such that
\begin{equation}\label{f:1}
\tilde x^TB_0\tilde x< \beta_0,\ \tilde x^TB_i\tilde x\le \beta_i, ~ i=1,\cdots,m.
\end{equation}
On the other hand, it follows from $(f)$ that there exists $t\ge0$ such that
\begin{equation}\label{f:2}
\tilde x^TB_0\tilde x+\sum_{i=1}^{m} t_i\tilde x^TB_i\tilde x\ge0\ge\beta_0+\sum_{i=1}^m t_i\beta_i.
\end{equation}
Combining \eqref{f:1} and \eqref{f:2} gives a contradiction.\qed
\end{proof}
\begin{remark}
Based on the similar analysis of Remark \ref{re:P}, Theorem \ref{th:g_slemma}  covers Theorems \ref{le:hs} and \ref{le:slemma}. Moreover, with the setting $m=2$, Theorem \ref{th:g_slemma} reduces to Polyak's extension of S-lemma, see \cite[Theorem 4.1]{Polyak98}.
\end{remark}
We present the following equality version of Theorem \ref{th:g_slemma} and omit the proof as it is similar to that of Theorem \ref{th:g_slemma}.
\begin{theorem} \label{th:g_slemma2}
Suppose that $n\ge 3,~ m\ge 1$ and $\{B_i: i=0, 1, \cdots, m\}$ can be generated by $\{A_1, A_2, A_3\}$ with $s_1A_1+s_2A_2+s_3A_3\succ0$  for some $s_1, s_2, s_3\in\R$. Assume that there exist $\bar x, \bar x'\in\R^n$ such that $\bar x^TB_i\bar x<\beta_i, ~i=1, \cdots, m$ and $\bar x'^TB_1\bar x'>\beta_1,\   \bar x'^TB_i\bar x'<\beta_i, ~i=2, \cdots, m$. Then, the following are equivalent:
\begin{itemize}
\item[(g)] $x^TB_1x=\beta_1,\  x^TB_ix\le\beta_i, i=2, \cdots, m~\Longrightarrow~ x^TB_0x\ge\beta_0$.
\item[(h)] $\exists\ t_1\in\R,\ t_2, \cdots, t_m\ge0$ such that
$
		B_0+\sum_{i=1}^m t_iB_i\succeq0,\
		\beta_0+\sum_{i=1}^m t_i\beta_i\le0.
$
\end{itemize}
\end{theorem}

S-lemma and its variants play a great role in revealing the hidden convexity of quadratic optimization problems, see \cite{Xia20} and references therein.
As applications of Theorems \ref{th:g_slemma} and \ref{th:g_slemma2}, we study the homogeneous quadratic optimization in $\R^n~(n\ge3)$ with two bilateral quadratic form constraints:
\begin{equation}\label{eq:Aqq2_1}\tag{HQPB}
	\begin{array}{cl}
		\min & x^TA_0x\\
		\ST   &
		m_1\le x^TA_1x\le M_1,\\
		& m_2\le x^TA_2x\le M_2,\\
	\end{array}
\end{equation}
where $A_i\in \mathcal{S}^{n}$, $i=0,1,2$, and  either there exist $s_0, s_1, s_2\in\R$ such that  $s_0A_0+s_1A_1+s_2A_2\succ0$ or ${\rm rank\ }(\{A_0, A_1, A_2\})<3$.
We assume $m_1<M_1$ and Slater condition holds, i.e., if  $m_2<M_2$, there exists $\widetilde{x}\in\R^n$ such that
\[
m_1< \widetilde{x}^TA_1\widetilde{x}< M_1,\  m_2< \widetilde{x}^TA_2\widetilde{x}< M_2,
\]
and otherwise  $m_2=M_2$, then
there exists $\bar x, \bar x'\in\R^n$ such that
\[
m_1< \bar x^TA_1\bar x< M_1,\ \bar x^TA_2\bar x< m_2;
\ m_1< \bar x'^{T}A_1\bar x'< M_1,\ \bar x'^{T}A_2\bar x'> m_2.
\]

Problem \eqref{eq:Aqq2_1} contains the following two special cases.
\begin{example}
Consider the generalized trust region subproblem with interval bounds \cite{wang15,Pong14}
\begin{equation}\label{eq:bGTRS}
	\min\{f_0(x):~m\le f_1(x)\le M\},
\end{equation}
where $f_i(x)=x^TA_ix+2a_i^Tx$, $i=0, 1$, $A_i\in\mathcal{S}^n$, $a_i\in\R^n$,  $i=0, 1$, $n\ge 2$ and $m<M$. We assume that there exists $s_0, s_1\in\R$ such that $s_0A_0+s_1A_1\succ0$ and there exists $\bar x\in\R^n$ such that $m< f_1(\bar x)< M$. This problem generalizes from the case $f_1(x)=\|x\|^2$ which dates back to Stern and Wolkowicz \cite{Stern95}.

Problem \eqref{eq:bGTRS} can be homogenized as the following problem with respect to $(x,t)\in\R^{n+1}$:
\begin{equation}\label{eq:HbGTRS}
	\begin{array}{cl}
		\min & x^TA_0x+2a_0^Txt\\
		\ST   &
		m\le x^TA_1x+2a_1^Txt\le M,\\
		&t^2=1.
	\end{array}
\end{equation}
It is not difficult to verify that the optimal values of problems \eqref{eq:bGTRS} and \eqref{eq:HbGTRS} are equal, and $t^*x^*$ globally solves problem \eqref{eq:bGTRS}  if and only if $(x^*, t^*)$ is a global minimizer of problem \eqref{eq:HbGTRS}.
One can also verify that problem \eqref{eq:HbGTRS} satisfies all the assumptions we made in \eqref{eq:Aqq2_1}.
\end{example}

\begin{example}
Consider the two-sided identital regularized total least squares problem \cite{Beck06}
\begin{equation}\label{eq:TRTLS}\tag{TRTLS}
\min \left\{\frac{{\|Ax-b\|}^2}{{\|x\|}^2+1}:~
 m\le \|x\|^2\le M \right\},
\end{equation}
where $A\in\mathcal{S}^n, ~b\in\R^n,  ~m<M, ~n\ge 2$, and $\|\cdot\|$ is the Euclidean norm. The standard Lagrangian duality may admit  a positive gap \eqref{eq:TRTLS}. Necessary and sufficient condition for the strong Lagrangian duality for \eqref{eq:TRTLS} is presented in  \cite{Yang20}.
By introducing the generalized Charnes-Cooper transformation
\[
y=\frac{x}{\sqrt{x^Tx+1}},\ z=\frac{1}{\sqrt{x^Tx+1}},
\]
\eqref{eq:TRTLS} is reformulated as
\[
	\begin{array}{cl}
		\min & y^TA^TAy-2b^TAyz+b^Tbz^2\\
		\ST   &
		mz^2\le y^Ty\le Mz^2 ~( \Longleftrightarrow~
\frac{1}{M+1}\le z^2\le \frac{1}{m+1}),\\
		&y^Ty+z^2=1,
	\end{array}
\]
which is a special case of  \eqref{eq:Aqq2_1} satisfying all assumptions we have made.
%
\end{example}


Suppose $m_2<M_2$.
Problem \eqref{eq:Aqq2_1} is equivalent to
\begin{eqnarray}
&\sup\{&t :~\{x\in\R^n:~x^TA_0x<t,-x^TA_1x \le -m_1, x^TA_1x\le M_1, \nonumber\\
&& ~~~~~~~~~~~~~~~~~~-x^TA_2x\le-\ m_2, x^TA_2x\le M_2\}=\emptyset\}\nonumber\\
=&\sup\{&t: ~(\lambda_1,\mu_1,\lambda_2, \mu_2)\ge0,~A_0+(\mu_1-\lambda_1)A_1+(\mu_2-\lambda_2)A_2\succeq0, \nonumber\\
&&~~~~~t-\lambda_1m_1+\mu_1M_1-\lambda_2m_2+\mu_2M_2\le0\} \nonumber\\
=&\sup\{&\lambda_1m_1-\mu_1M_1+\lambda_2m_2-\mu_2M_2: ~(\lambda_1,\mu_1,\lambda_2, \mu_2)\ge0, \nonumber\\ &&A_0+(\mu_1-\lambda_1)A_1+(\mu_2-\lambda_2)A_2\succeq0
	\},\label{eq:ddd}
\end{eqnarray}
where the first equation follows from Theorem \ref{th:g_slemma}.  One can verify that \eqref{eq:ddd} is the Lagrange dual problem of \eqref{eq:Aqq2_1}.  The other case $m_2=M_2$ is similarly analyzed based on
Theorem \ref{th:g_slemma2}.
Thus, we have proved the following result.
\begin{theorem}
Strong duality holds for \eqref{eq:Aqq2_1} under suitable assumptions.
\end{theorem}

\section{Conclusion}
We extend  Polyak's  theorem on the convexity of the joint numerical range from  three to  arbitrary number of quadratic forms, provided that their Hessian matrices can be generated by three matrices with a  positive definite linear combination. It also covers the classical Dines's theorem. As applications, we further extend Yuan's lemma and S-lemma to the system with more than three quadratic forms.  With the help of our extended Yuan's lemma, we establish a more general assumption under which the standard second-order necessary condition holds at local minimizer. It is unknown whether our assumption can be further relaxed. The further extended S-lemma helps to reveal strong duality of homogeneous quadratic optimization problem with two bilateral constraints, which includes the generalized trust region subproblem with interval bounds and the two-sided regularized total least squares problem as special cases. Future works include more applications of our new results and further extensions to nonhomogeneous alternative theorems.

\bibliographystyle{spmpsci_unsrt}
\bibliography{references}

%
%


%
%

\end{document}